\documentclass
{amsart}

\usepackage{amssymb,mathrsfs}
\usepackage{amsmath}
\usepackage{amsfonts, mathrsfs}
\usepackage{amsthm}
\usepackage{mathabx}
\usepackage{enumitem}
\usepackage
{hyperref}
\usepackage{amsmath,amsthm,amsfonts,amssymb,mathtools}

\newcommand{\Be}{\begin{equation}}
\newcommand{\Ee}{\end{equation}}
\newcommand{\Bea}{\begin{eqnarray}}
\newcommand{\Eea}{\end{eqnarray}}
\newcommand{\Bel}{\begin{align}}
\newcommand{\Eel}{\end{align}}
\newcommand{\Beas}{\begin{eqnarray*}}
\newcommand{\Eeas}{\end{eqnarray*}}
\newcommand{\Benu}{\begin{enumerate}}
\newcommand{\Eenu}{\end{enumerate}}
\newcommand{\Bi}{\begin{itemize}}
\newcommand{\Ei}{\end{itemize}}

\newcommand\supp{\operatorname{supp}}

\def\R{{\mathbb R}}
\def\Z{{\mathbb Z}}

\def\C{{\mathbb C}}

\def\H{{\mathbb H}}

\newcommand{\pj}{\mathcal P_{\!j}}

\theoremstyle{plain}
\newtheorem{thm}{Theorem}[section]

\newtheorem{lem}[thm]{Lemma}
\newtheorem{prop}[thm]{Proposition}

\theoremstyle{remark}
  
\theoremstyle{definition}
\newtheorem*{defn}{Definition}

\numberwithin{equation}{section}

\newcommand{\astt}{\ast_{\mathbb H}}
\usepackage[utf8]{inputenc}
\usepackage{color,soul}

\setul{0.5ex}{0.3ex}
\setulcolor{red}

\begin{document}

\title[]{$L^p-L^q$ estimates for the circular maximal operator on  Heisenberg radial functions
}
\author[Juyoung Lee]{Juyoung Lee}
\author[Sanghyuk Lee]{Sanghyuk Lee}

\address{Department of Mathematical Sciences and RIM, Seoul National University, Seoul 08826, Republic of  Korea}
\email{shklee@snu.ac.kr}
\email{ljy219@snu.ac.kr } 


\begin{abstract}
$L^p$ boundedness of the circular maximal function  $\mathcal M_{\H^1}$  on the Heisenberg group $\H^1$ has received considerable attentions. While the problem  still remains 
open,  $L^p$ boundedness of $\mathcal M_{\H^1}$ on Heisenberg radial functions  was recently shown for $p>2$ by Beltran, Guo, Hickman, and Seeger \cite{BGHS}.  
In this paper we extend their result considering  the local maximal operator   $M_{\H^1}$ which is defined by  taking  supremum over $1<t<2$. We  prove  $L^p$--$L^q$ estimates for $M_{\H^1}$ on Heisenberg radial functions on the optimal range of $p,q$  modulo the borderline cases.  Our argument also provides a simpler  proof of the aforementioned  result due to Beltran et al.  
\end{abstract}
\maketitle

\section{introduction}
For $d\ge 2$ the spherical maximal function is given by
\[ \mathcal M_{\R^d} f(x)=\sup_{t>0}\left| \frac{1}{\sigma(\mathbb{S}^{d-1})} \int_{\mathbb S^{n-1}}f(x-ty)d\sigma(y)\right|, \]
where $\mathbb{S}^{d-1}\subset \mathbb R^d$ is the $(d-1)$-dimensional sphere centered at the origin and $d\sigma$ is the surface measure on $\mathbb S^{d-1}$.
When $d\ge 3$, it was shown by 
Stein \cite{Stein2}  that  $\mathcal M_{\R^d} f$ is bounded on $L^p$ if and only if $p>\frac{d}{d-1}$.  The case $d=2$ was later settled by   Bourgain \cite{B2}. 
An alternative proof of Bourgain's result was subsequently found by  Mockenhaupt, Seeger, Sogge \cite{MSS2}, who used a  local smoothing estimate for the wave operator.  
We now consider the local maximal operator 
\[  M_{\R^d}f(x)=\sup_{1<t<2}\left|\int_{\mathbb S^{d-1}}f(x-ty)d\sigma(y)\right|. \]
As is easy to see, the maximal operator $\mathcal M_{\R^d}$ can not be bounded from $L^p$ to $L^q$ unless  $p=q$. However,  
$M_{\R^d}$ is bounded from $L^p$ to $L^q$ for some $p<q$   thanks to the supremum taken over  the restricted range $[1,2]$. This phenomenon is called  \emph{$L^p$ improving}.  
Almost complete characterization of  $L^p$ improving property of  $M_{\R^2}$  was obtained by Schlag  \cite{Schlag97} except for the endpoint cases. A different proof  which is based on 
$L^p$--$L^q_\alpha$ smoothing estimate for the wave operator  was also found  by  Schlag and Sogge \cite{SS}. They also 
proved  $L^p$--$L^q$ boundedness  of $\mathcal M_{\R^d}$ for $d\geq 3$ which  is optimal up to the borderline cases.  Most of  the left open endpoint cases were settled by the second author \cite{L} but 
there are some endpoint cases where $L^p$--$L^q$ estimate remains unknown though restricted weak type bounds are available for such cases.  
There are results which extend the aforementioned results  to variable coefficient settings, see \cite{Sogge, SS}. Also, see   \cite{AHRS,  RS,  BORSS} and references therein for
recent extensions  of the earlier results. 

%
%

The analogous spherical maximal operators on the Heisenberg group $\H^n$ also have attracted considerable interests. 
The Heisenberg group $\H^n$ can be identified by $\R^{2n}\times \R$ with the  noncommutative multiplication law 
\[ (x,x_{2n+1})\cdot (y,y_{2n+1})=(x+y,x_{2n+1}+y_{2n+1}+x\cdot Ay), \]
where  $(x,x_{2n+1})\in \mathbb R^{2n}\times \mathbb R$ and $A$ is the $2n\times 2n$ matrix  given by \[A=\begin{pmatrix}
0 & -I_n\\
I_n & 0
\end{pmatrix}.\] The natural dilation structure on $\mathbb H^n$ is $t(x,x_{2n+1})=(tx, t^2x_{2n+1})$. 
Abusing the notation,  since there is no ambiguity, we denote by $d\sigma$ the usual surface measure of $\mathbb S^{2n-1}\times \lbrace 0\rbrace$. Then, the dilation $d\sigma_t$ of the measure 
$d\sigma$ is defined by $\langle f, d\sigma_t\rangle=\langle f(t\cdot),d\sigma\rangle$. Thus,  the average over the sphere is now given by  
\[ f\astt d\sigma_t (x,x_{2n+1})=\int_{\mathbb S^{2n-1}}f(x-ty,x_{2n+1}-tx\cdot Ay) d\sigma(y). \]
We consider the associated local  spherical maximal operator
\[ M_{\H^n}f(x,x_{2n+1})=\sup_{1<t<2}\left| f\astt d\sigma_t (x,x_{2n+1}) \right|. \]
 Similarly, the global maximal operator $\mathcal M_{\H^n}$ is defined by taking supremum over $t>0$. As in the Euclidean case,  $L^p$ boundedness of $M_{\H^n}$  is essentially equivalent to 
 that of $\mathcal M_{\H^n}$ (for example, see \cite{BGHS} or Section \ref{g-max}).  The spherical maximal operator on $\H^n$ was first studied  by Nevo and Thangavelu \cite{NeT}. 
  It is easy to see that $ M_{\H^n}$  is bounded on $L^p$ only if $p>\frac{2n}{2n-1}$  by using Stein's  example  (\cite{Stein2})
\[ f(x,x_{2n+1})=\vert x\vert^{1-2n}  \log \frac{1}{\vert x\vert} \,\phi_0(x,x_{2n+1}) \]
for a suitable cutoff function $\phi_0$ supported near the origin. For $n\geq 2$,  $L^p$ boundedness of  $ \mathcal M_{\H^n}$ on the optimal range  was  independently proved by M\"uller and Seeger \cite{MS},  and by Narayanan and Thangavelu \cite{NaT}.  
Furthermore,   for $n\geq 2$,  Roos, Seeger and Srivastava\cite{RSS} recently obtained  the complete $L^p$--$L^q$ estimate for $M_{\H^n}$ except for some endpoint cases.
Also see \cite{K} for related results. 

However, the problem still remains open when $n=1$.   

\begin{defn} We say 
a function $f: \H^1\rightarrow \C$ is Heisenberg radial  if
$f(x,x_3)=f(Rx,x_3)$
for all $R\in{\rm SO}(2)$.
\end{defn}

\noindent 
Beltran, Guo, Hickman and Seeger \cite{BGHS} obtained  $L^p$ boundedness of   $M_{\H^1}$ on the Heisenberg radial functions for $p>2$. 
In the perspective of  the results concerning the local maximal operators (\cite{Schlag97, SS, L, RSS}), it  is natural to consider  $L^p$--$L^q$ estimate for $M_{\H^1}$.  
The main result of this paper  is the following which completely characterizes  $L^p$ improving property of $M_{\H^1}$ on Heisenberg radial function except for some borderline cases. 

\begin{thm}\label{main-thm}
Let $P_0=(0,0), P_1=(1/2,1/2),$ and $P_2=(3/7,2/7)$, and let $\mathbf{T}$ be the closed region bounded by the triangle $\Delta P_0P_1P_2$. 
Suppose $(1/p,1/q)\in \lbrace P_0 \rbrace\cup\mathbf{T}\setminus (\overline{P_1P_2}\cup \overline{P_0P_2})$.   Then, the estimate 
\Be 
\label{max-pq} \Vert M_{\H^1}f\Vert_q\lesssim \Vert f\Vert_{L^p} \Ee
holds for any Heisenberg radial  function $f$. Conversely, if $(1/p,1/q)\notin \mathbf{T}$, then the estimate fails.
\end{thm}

Though the Heisenberg radial assumption simplifies the structure of the averaging operator significantly, the associated defining function of the averaging operator  is still lacking of curvature properties. In fact, the defining function has vanishing rotational and  cinematic curvatures at some points, see  \cite{BGHS} for detailed discussion.  This increases  complexity of the problem.  To overcome the issue of vanishing curvatures,   Beltran, et al. \cite{BGHS}  used the oscillatory integral operators with two-sided fold singularities and the variable coefficient version of local smoothing estimate (\cite{BHS}) combined with  additional localization.

The approach in this paper  is quite different from that in \cite{BGHS}.  Capitalizing on the Heisenberg radial assumption, we make a change of variables so that the averaging operator on the Heisenberg radial function takes a form close to the circular average, see \eqref{def-A1} below.  While the defining function of the consequent operator  still does not have nonvanishing  rotational and  cinematic curvatures, via a further change of variables we can apply the $L^p$--$L^q$ local  smoothing estimate (see,  Theorem  \ref{local-pq} below) in a more straightforward manner by exploiting the apparent connection to the wave operator (see \eqref{def-A} and  \eqref{def-A2}). Consequently, our approach  also
provides a simplified proof of the result in \cite{BGHS}.  See Section  \ref{g-max}. 

Even though we use the local smoothing estimate, we do not need to use the  full strength of the local smoothing estimate  in $d=2$ since we only need the 
sharp $L^p$--$L^q$ local smoothing estimates for $(p,q)$ near $(7/3, 7/2)$.  Such  estimates can also be obtained by interpolation and  scaling argument 
if one uses the trilinear restriction estimates for the cone and the sharp local smoothing estimate for some large $p$ (for example, see \cite{LV}). 

We close the introduction showing the necessity part of  Theorem \ref{main-thm}.

\noindent{\it Optimality of $p,q$ range.} 
We show
\eqref{max-pq} implies   $(1/p,1/q)\in \mathbf{T}$, that is to say, 
\[  {\bf (a)}\  p\leq q ,  \  \  \   {\bf (b)} \ 1+1/{q}\geq {3}/{p} , \  \  \  {\bf (c)}\ {3}/{q}\geq {2}/{p}.\]
To see  ${\bf (a)}$,  let $f_R$ be the characteristic function of a ball of radius $R\gg1$, centered at $0$. Then, $M_{\H^1}f_R$ is also supported in a ball $B$ of radius $\sim R$ and 
$M_{\H^1}f_R\gtrsim 1$ on $B$.  Thus, $\sup_{R>1}{\Vert M_{\H^1}f_R\Vert_q}/{\Vert f_R\Vert_p}$ is finite  only if $p\leq q$.
For ${\bf (b)}$  let $g_r$ be the characteristic function of a ball of radius $r\ll 1$ centered at $0$. Then,  $\vert M_{\H^1}g_r(x,x_3) \vert \gtrsim r$ when  $(x,x_3)$ is contained in a $c_0r-$neighborhood of $\lbrace (x,x_3):1<\vert x\vert<2, x_3=0 \rbrace$ for a small constant $c_0>0$. Thus, \eqref{max-pq} implies 
$ r^{1+{1}/{q}}\lesssim r^{{3}/{p}}$,
which gives $1+{1}/{q}\geq {3}/{p}$ if we let $r\to 0$.
Finally, to show ${\bf (c)}$ we consider  $h_r$ which is the characteristic function of an $r-$neighborhood of
$\lbrace (x,x_3):\vert x\vert =1, x_3=0 \rbrace$ with $r\ll 1$. Then, $\vert M_{\H^1}h_r(x,x_3)\vert \gtrsim c>0$ when $(x,x_3)$ is in an $r-$ball centered at $0$. Thus, \eqref{max-pq}  gives 
$r^{3/{q}}\lesssim r^{2/{p}}$, 
which gives $3/{q}\geq 2/{p}$.

The maximal estimate \eqref{max-pq} for general $L^p$ functions  has a smaller range of $p,q$. Let $h_r$ be a characteristic function of the set $\lbrace (x,x_3) : |x_1-1|<r^2, |x_2|<r,|x_3|<r \rbrace$  
for a sufficiently small $r>0$.  Then $M_{\H^1}h_r (x,x_3)\sim r$  if  $-1\leq x_1\leq 0, |x_2|< cr, |x_3|<cr$ 
for a small constant  $c>0$ independent of $r$. Thus, \eqref{max-pq} implies  $r^{1+2/{q}}\lesssim  r^{{4}/{p}}$. 
It seems to be plausible to conjecture that  \eqref{max-pq} holds for general $f$ as long as  $1+ {2}/{q}-4/{p}\ge 0$, $3/{q}\ge {2}/{p}$, and $1/{q}\leq1/{p}$. So, the range of $p,q$ is properly contained in  $\mathbf{T}$.  

%
%
%
\section{Proof of Theorem \ref{main-thm}}
In this section we prove  Theorem \ref{main-thm} while assuming Proposition \ref{goal2} and Proposition  \ref{goal4} (see below), which we show in the next section.

\subsection{Heisenberg Radial function} 
Since $f$ is a Heisenberg radial  function,  we have $ f(x,x_3)=f_0(|x|,x_3)$ for some  $f_0$. Let 
us set 
\[ g(s,z)=f_0(\sqrt{2s}, x_3), \quad s\ge 0. \] 
Then, it follows  $f(x ,x_3)=g(|x|^2/{2},x_3)$. 
 Since  $f\astt d\sigma_t (r,0,x_3)=\int f(r-ty_1,-ty_2,x_3-try_2)d\sigma(y)=\int g(\frac{r^2+t^2}{2}-t ry_1,x_3-try_2)d\sigma(y)$,   we have
\begin{align}
\label{def-A1}
f\astt d\sigma_t (r,0,x_3) =  g\ast  d\sigma_{tr} \Big(\frac{r^2+t^2}{2},x_3\Big). 
 \end{align}
 Let us define an operator $\mathcal A_t$  by 
 \Be 
 \label{def-A}
 \mathcal A_t g(r,x_3)= \frac1{(2\pi)^2}\int_{\mathbb R^2} e^{i(\frac{r^2+t^2}{2}\xi_1+x_3\xi_2)} \widehat{d\sigma}(tr\xi)  \,  \widehat{g}(\xi) d\xi.
 \Ee
 Using Fourier inversion, we have
 \Be
 \label{def-A2} f\astt d\sigma_t (r,0,x_3)=\mathcal A_t g(r,x_3).
 \Ee
 Since  $f\astt d\sigma_t $ is also Heisenberg radial,\footnote{This is true because $\rm{SO}(2)$ is an abelian group. However,  $\rm{SO}(n)$ is not commutative in general, so 
 the property is  not valid in higher dimensions.} 
$ \Vert M_{\H^1}f\Vert_q^q = \int |M_{\H^1} f(r,0,x_3)|^q rdrdx_3. $ 
A computation shows $\|f\|_{L^p_{x, x_3}}=\|g\|_{L^p_{r,x_3}}$. 
Therefore, we see that  the estimate \eqref{max-pq} is equivalent to  
\Be  
\label{max-A} 
\big\| r^\frac1q \sup_{1<t<2} |\mathcal A_t g |  \big\|_{L^q_{r,x_3}}\le  C \|g\|_p. 
\Ee
\newcommand{\wchi}{\widetilde \phi} 
In what follows we show  \eqref{max-A} holds for 
$p,q$ satisfying 
\Be
\label{tri-con}
p\le q, \   3/p-1/q<1,  \  1/{p}+2/{q}> 1. 
\Ee
Then, interpolation with the trivial $L^\infty$ estimate proves Theorem \ref{main-thm}.

\subsection{Decomposition}
Let $\phi$ denote a positive smooth function on $\R$ supported in $[1-10^{-3}, 2	+10^{-3}]$ such that $\sum_{j=-\infty}^\infty \phi(s/2^j)=1$ for $s>0$.  We set  $\phi_j(s)=\phi(s/2^j)$. 
To show \eqref{max-A} we decompose  $\mathcal A_t$   as follows:
\[\mathcal A_t g(r,x_3)
=   \sum_{k\in\Z} \phi_k(r)  \mathcal   A_t g(r,x_3).\] 

 We decompose $g$ via the Littlewood-Paley decomposition and  
try to obtain estimates for each decomposed pieces.   For the purpose
 we  denote $\phi_{<\ell}=\sum_{j<\ell }\phi_j$ and $\phi_{\,\ge \ell}=\sum_{j\ge \ell }\phi_j$ and define 
 the projection operators 
\[  \widehat{\mathcal P_j g} (\xi):=\phi_j(|\xi|) \widehat g(\xi), \quad  \widehat{\mathcal P_{<j}  g} (\xi):=\phi_{<j}(|\xi|) \widehat g(\xi).
\] 

 Our proof of \eqref{max-A} mainly relies  on the following two propositions, which we prove in Section \ref{pf-pros}.

\begin{prop}\label{goal2}
Let $|k|\ge 2$ and $j\geq -k$. Suppose  
\Be 
\label{pq}
p\leq q, \  1/{p}+1/{q}\leq 1, \  1/{p}+3/{q}\geq 1. 
\Ee
 Then,  for $\epsilon>0$ we have
\begin{equation}\label{goal2 ineq}
\Big\Vert \sup_{1<t<2}  | \phi_k(r)\mathcal A_t  \pj g|  \Big \Vert_{L^q_{r,x_3}}
\lesssim 
\begin{cases}
2^{(j+k)(\frac{3}{2p}-\frac{1}{2q}-\frac{1}{2}+\epsilon)+\frac{k}{q}-\frac{2k}{p}}\Vert g\Vert_{L^p},  & \  \  \ k\ge 2,
\\[3pt]
2^{(j+k)(\frac{3}{2p}-\frac{1}{2q}-\frac{1}{2}+\epsilon)+\frac{2k}{q}-\frac{2k}{p}}\Vert g\Vert_{L^p}, & \  k< -2.
\end{cases}
\end{equation}
\end{prop}

The estimate \eqref{goal2 ineq} continues to be valid for the  case  $k=-1,0,1$. However, 
the range of $p,q$ for which \eqref{goal2 ineq} holds  gets smaller.  

\begin{prop}\label{goal4}
Let $j\geq -1$ and  $k=-1,0,1$. Suppose  $p\leq q$, $1/{p}+1/{q}< 1$ and $1/{p}+2/{q}> 1$. Then,   for $\epsilon>0$ we have
\[ \Big\Vert \sup_{1<t<2}  | \phi_k(r)\mathcal A_t  \pj g|  \Big \Vert_{L^q_{r,x_3}} \lesssim   2^{ \frac{j}{2}(\frac{3}{p}-\frac{1}{q}-1)+\epsilon j}\Vert g\Vert_{L^p}. \]
\end{prop}

We frequently use the following elementary lemma (for example,  see \cite{L}) which plays the role of the Sobolev imbedding theorem.

\begin{lem}\label{sobo}
Let $I$ be an interval and let $F$ be a smooth function defined on $\mathbb{R}^n\times I$. Then, for $1\le p\leq \infty$,
\[ \Big\Vert\sup_{t\in I}\vert F(x,t)\vert\Big\Vert_{L^p(\mathbb{R}^n)}\lesssim   {|I|^{-\frac1p}}\Vert F\Vert_{L^p(\mathbb{R}^n\times I)}+\Vert F\Vert_{L^p(\mathbb{R}^n\times I)}^{\frac{(p-1)}{p}}\Vert\partial_tF\Vert_{L^p(\mathbb{R}^n\times I)}^{\frac{1}{p}}. \]
\end{lem}

\subsection{Proof of   \eqref{max-A}
}
We prove \eqref{max-A} considering the three cases  $  k\le -2,$ $  |k|\le 1,$ and $ k\ge 2$, 
  separately.  In fact, we make  use of the change of variables 
\eqref{chg} to apply the local smoothing estimate for the wave propagator (see Section \ref{ls}).
Since $1<t<2$,  $|\det\frac{\partial (y_1, y_2, \tau)}{\partial  (r,x_3,t)}|= |r^2-t^2|\sim \max(2^{2k}, 1)$ if $|k|\ge 2$. 
The cases  $|k|\ge 2$  can be handled in a rather straightforward manner. However, 
the Jacobian may vanishes when  $ |k|\le 1$, so the  map $(r,x_3,t)\to (y_1, y_2, \tau)$ becomes singular. 
This requires further decomposition away from the set $\{r=t\}$. See Section \ref{singular section}. 
This is why we separately consider the three cases.

\subsection*{Case $k\le - 2$}
We claim  that 
\Be 
\label{claim-2} 
\Big\Vert r^{\frac{1}{q}}  \sum_{k\le -2}  \sup_{1<t<2}  | \phi_k(r)\mathcal A_t  g| \Big\Vert_{L^q_{r,x_3}}\lesssim \Vert g\Vert_{L^p} \Ee 
holds provided that  $p,q$ satisfy $2/p<3/q$, $3/p-1/q<1$, and \eqref{pq}. 
Thus \eqref{claim-2} holds for $p,q$ satisfying \eqref{tri-con}.

Let us set $g_k=\mathcal P_{<-k}\,g$ and $g^k= g- \mathcal P_{< -k}g$ so that $g=g_k+g^k$. 
We break 
\Be\label{g-decomp}   \phi_k(r)\mathcal A_t  g=   \phi_k(r)\mathcal A_t g_k +  \phi_k(r)\mathcal A_t  g^k.  \Ee
We first consider $\phi_k(r)\mathcal A_t  g_k $. We shall show  that 
\Be 
\label{ho} \Big\Vert r^{\frac{1}{q}} \sup_{1<t<2}  | \phi_k(r)\mathcal A_t g_k | \Big\Vert_{L^q_{r,x_3}} \lesssim 2^{\frac{3k}{q}-\frac{2k}{p}}\Vert g\Vert_{L^p}
\Ee
holds for $1\le p\le q\le \infty$.  We recall \eqref{def-A} and note that  $\partial_t(\widehat{d\sigma}(tr\xi))$ is uniformly bounded because $|r\xi|\lesssim 1$. 
Since $\supp \widehat{g_k }\subset\{\xi: |\xi|\le C2^{-k}\}$ and $\partial_t e^{\frac{r^2+t^2}{2}\xi_1}=t\xi_1 e^{\frac{r^2+t^2}{2}\xi_1}$,  we have $ \| \phi_k(r)  \partial_t  \mathcal A_t g_k\|_q\lesssim  2^{-k}\| \phi_k(r)\mathcal A_t  g_k\|_q$ by the Mikhlin multiplier theorem. 
Applying Lemma \ref{sobo} to  $ \phi_k(r)\mathcal A_t  g_k$, we see that \eqref{ho} follows if we show
\Be 
\label{ho2}
 \Vert  \phi_k(r)\mathcal A_t  g_k \Vert_{L^q_{r,x_3,t}(\mathbb R^2\times [1,2])}\lesssim 2^{\frac{3k}{q}-\frac{2k}{p}}\Vert g\Vert_{L^p}.
  \Ee
  
 We now make use of the change   variables 
\Be
\label{chg}    (r,x_3,t)\to  (y_1, y_2, \tau):=\left(\frac{r^2+t^2}{2},x_3,rt  \right).  \Ee
Note that  
\Be 
\label{jacobian}
\det \frac{\partial (y_1, y_2, \tau)}{\partial  (r,x_3,t)}= r^2-t^2.
\Ee 
Since $k\le -2$ and $t\in [1,2]$, we have $|\det\frac{\partial (y_1, y_2, \tau)}{\partial  (r,x_3,t)}|\sim 1$.  Thus the left hand side of \eqref{ho2} is bounded by 
\[ 
 C\Big\Vert \phi_k(r(y_1, y_2, \tau))\int e^{iy\cdot \xi}\, \widehat{g}(\xi)\widehat{d\sigma}(\tau\xi)\phi_{<-k}(\xi)d\xi \Big\Vert_{L^q_{y,\tau}(\R^2\times [2^{-1},2^2])}. \]
Changing variables $\xi\to 2^{-k}\xi$ and $(y, \tau)\to (2^ky, 2^k\tau)$ gives
\[  \Vert  \phi_k(r)\mathcal A_t  g_k \Vert_{L^q_{r,x_3,t}(\mathbb R^2\times [1,2])} \lesssim 2^{\frac{3k}{q}}\Big\Vert \int e^{i y\cdot \xi}\, \mathfrak m(\xi)  \widehat{g(2^k\cdot)} (\xi)  d\xi \Big\Vert_{L^q_{y,\tau}(\R^2\times [2^{-1},2^2])}, \]
where $\mathfrak m(\xi)= \widehat{d\sigma}( \tau\xi)\phi_{<0}(\xi)$. 
Since $\tau\sim 1$ and $\phi_{<0}(\xi)$ is a smooth function supported in the set $\lbrace \xi : |\xi|\lesssim 1 \rbrace$,  $\mathfrak m(\xi)$ is a  smooth multiplier whose derivatives are uniformly bounded.  
So, the multiplier operator  given  by $\mathfrak m$ is uniformly bounded from $L^p(\mathbb R^2)$ to $L^q(\mathbb R^2)$ for $\tau\in  [2^{-1},2^2]$. Thus, via scaling 
we obtain \eqref{ho2} and, hence, \eqref{ho}.

Using  the triangle inequality and \eqref{ho}, we have 
\[ 
 \Big\Vert r^{\frac{1}{q}}  \sup_{1<t<2}  \sum_{k\le -2}  |  \phi_k(r)\mathcal A_t  g_k | \Big\Vert_{L^q_{r,x_3}} \lesssim  
\Big(\sum_{k\le -  2} 2^{\frac{3k}{q}-\frac{2k}{p}}\Big )  \|g\|_p \lesssim \|g\|_p
\]
because $2/p<3/q$. We now consider $\phi_k(r)\mathcal A_t  g^k$ for which we use Proposition \ref{goal2}.  Since 
\[   \Big\Vert r^{\frac{1}{q}}  \sup_{1<t<2}  \sum_{k\le -2} |  \phi_k(r)\mathcal A_t  g^k | \Big\Vert_{L^q_{r,x_3}}   \le   \sum_{k\le -2} \sum_{j\geq -k}\Big\Vert r^{\frac{1}{q}}  \sup_{1<t<2}  | \phi_k(r)\mathcal A_t  \pj g | \Big\Vert_{L^q_{r,x_3}} \] 
and since $p,q$ satisfy  $3/p-1/q<1$, $2/p<3/q$, and \eqref{pq},
using the estimate \eqref{goal2 ineq}, we get 
\[
 \Big\Vert r^{\frac{1}{q}}  \sup_{1<t<2}   \sum_{k\le -2} |  \phi_k(r)\mathcal A_t  g^k | \Big\Vert_{L^q_{r,x_3}} \lesssim  
\Big(\sum_{k\le -  2} 2^{\frac{3k}{q}-\frac{2k}{p}}\Big )  \|g\|_p \lesssim \|g\|_p.  \]
Combining this with the above estimate  for $g\to \phi_k(r)\mathcal A_t  g^k$  gives \eqref{claim-2} and this  proves the claim.

\subsection*{Case $k\ge 2$}  
In this case we show
\Be
\label{claim-3} 
\Big\Vert r^{\frac{1}{q}}  \sum_{k\ge 2}  \sup_{1<t<2}  | \phi_k(r)\mathcal A_t  g| \Big\Vert_{L^q_{r,x_3}}\lesssim \Vert g\Vert_{L^p} 
\Ee
if $p\le q$, $3/p-1/q<1$, and \eqref{pq} holds. So,  we have \eqref{claim-3} if \eqref{tri-con} holds.

In order to prove \eqref{claim-3}  we first prove the following.  

\begin{lem}
\label{lem:kernel} Let $k\ge -1$. If  $|t|\lesssim 1$ and $0\le s\lesssim 2^{2k}$,  then 
\Be  
\label{kernel}
|\mathcal A_t {\mathcal P}_{<-k}g | (\sqrt{2s}, x_3)\lesssim  \mathcal E^N_{k}\ast | g|(s, x_3), \Ee
where $ \mathcal E^N_\ell (y) =2^{-2\ell}(1+ 2^{-\ell}|y| )^{-N}.$
\end{lem}

\begin{proof}
We note that
\begin{align*}
 \mathcal A_t  {\mathcal P}_{<-k}g(\sqrt{2s}, x_3)
& = K\ast g  \big(s+2^{-1}t^2,  x_3 \big ),
\end{align*}
where 
\[  K(y)=   \frac1{(2\pi)^2}\int e^{iy\cdot \xi}\phi_{<-k}(\xi)\widehat{d\sigma}(t\sqrt{2s}\xi) d\xi.\] 
We note  $\partial_{\xi}^\alpha [\phi_{<-k}(2^{-k}\xi)\widehat{d\sigma}(2^{-k} t\sqrt{2s}\xi) ]=O(1)$ since $s\lesssim 2^{2k}$. 
Thus, changing variables $\xi\to 2^{-k}\xi$, by integration by parts we have
$|K|\lesssim  \mathcal E^N_{k} $ for any $N>0$.  Since $|t|\lesssim 1$ and $k\ge -1$, we see 
$\mathcal E^N_{k} (y_1+2^{-1}t^2, y_2)\lesssim \mathcal E^N_{k} (y_1, y_2)$. 
  Therefore,  we get \eqref{kernel}. 
  \end{proof}

\begin{proof}[Proof of \eqref{claim-3}] 
We begin by observing a localization property of the  operator $\mathcal A_t$. 
From \eqref{def-A1} 
we note that 
\[  \frac{r^2+t^2}{2}-t ry_1 \subset I_k:= [2^{2k-1}( 1- 10^{-2}),2^{2k+1}( 1+10^{-2})]  \] 
for $r\in \supp \phi_k$   if $k$ is large enough, i.e., $2^{-k}\le    10^{-3}$. Thus,  from \eqref{def-A1}   and  \eqref{def-A2}  we see  that 
\Be
\label{localization}
  \phi_k(r) \mathcal A_t g(r,x_3)=   \phi_k(r)\mathcal A_t ([g]_k)(r,x_3) 
  \Ee
where $[g]_k(r,x_3)=\chi_{I_k}(r) g(r,x_3)$.   
Clearly,  the intervals ${I_k}$ are finitely overlapping  and  so are the supports of $\phi_k$. 
Since $p\le q$, by a standard localization argument  it is sufficient for  \eqref{claim-3} to show 
\Be 
\label{haa}  \Big\Vert r^{\frac{1}{q}}   \sup_{1<t<2}  | \phi_k(r)\mathcal A_t  g| \Big\Vert_{L^q_{r,x_3}}\lesssim \Vert g\Vert_{L^p}
\Ee
for $k\ge 2$.

Using the decomposition \eqref{g-decomp}, we first consider $ \phi_k(r)\mathcal A_t  g^k$.    
Since 
\Be
\label{tri}
  \Big\Vert r^{\frac{1}{q}}  \sup_{1<t<2}  |   \phi_k(r)\mathcal A_t  g^k | \Big\Vert_{L^q_{r,x_3}}   \le  \sum_{j\geq -k}\Big\Vert r^{\frac{1}{q}}  \sup_{1<t<2}  | \phi_k(r)\mathcal A_t  \pj g | \Big\Vert_{L^q_{r,x_3}}
 \Ee  and since 
$3/p-1/q<1$, $p\le q$, and \eqref{pq} holds,  
using the estimate \eqref{goal2 ineq}, we get 
\[\Big\Vert r^{\frac{1}{q}}  \sup_{1<t<2}  |  \phi_k(r)\mathcal A_t  g^k | \Big\Vert_{L^q_{r,x_3}} \lesssim  
 2^{\frac{2k}{q}-\frac{2k}{p}} \|g\|_p \lesssim \|g\|_p.  \]
 We now handle $ \phi_k(r)\mathcal A_t  g_k$.  
Changing variables $r\mapsto \sqrt{2s}$,  we have
\begin{align*}
\Big\Vert r^\frac1q \sup_{1<t<2} |\phi_k(r)\mathcal A_t  g_k |\Big\Vert_{L^q_{r,x_3}}^q
\lesssim \int \phi_k(\sqrt{2s}) \left(\sup_{1<t<2}  |\mathcal A_t  g_k(\sqrt{2s}, x_3) |  \right)^qdsdx_3.
\end{align*}
Since $1<t<2$, $k\ge 2$, and $ g_k=\mathcal P_{<-k} g$,  by Lemma \ref{lem:kernel}
$|\mathcal A_t  g_k(\sqrt{2s}, x_3) |\lesssim  \mathcal E^N_{k}\ast | g|(s, x_3)$. Hence, 
\[  \Big\Vert r^\frac1q \sup_{1<t<2} |\phi_k(r)\mathcal A_t  g_k |\Big\Vert_{L^q_{r,x_3}} \lesssim  \| \mathcal E^N_{k}\ast | g| \Vert_{L^q_{s,x_3}} \lesssim 2^{2k(1/q-1/p)} \|g\|_p \le \|g\|_p .\]
The second inequality follows by Young's convolution inequality and the third is clear because $k\ge 2$ and $p\le q$. Therefore, we get \eqref{haa}. 
\end{proof}

\subsection{Case $|k|\le 1$}  
To complete the proof of \eqref{max-A}, the matter is now reduced to obtaining  
\[
\Big\Vert  r^\frac1q \sup_{1<t<2}  | \phi_k(r)\mathcal A_t  g| \Big\Vert_{L^q_{r,x_3}}\lesssim \Vert g\Vert_{L^p}, \quad  k=-1,0,1 
\]
if $p,q$ satisfy  \eqref{tri-con}. 
 In order to show this we use Proposition \ref{goal4}.  Using the decomposition \eqref{g-decomp}, we first consider $ \phi_k(r)\mathcal A_t  g^k$.  
 Note that \eqref{pq} is satisfied  if \eqref{tri-con} holds. 
Since  $3/p-1/q<1$,  by \eqref{tri} and Proposition \ref{goal4} we see  
 \[\Big\Vert r^{\frac{1}{q}}  \sup_{1<t<2}  |  \phi_k(r)\mathcal A_t  g^k | \Big\Vert_{L^q_{r,x_3}} \lesssim  
  \sum_{j\ge -k} 2^{ \frac{j}{2}(\frac{3}{p}-\frac{1}{q}-1)+\epsilon j}\Vert g\Vert_{L^p}  \lesssim \|g\|_p\]
 taking a small enough $\epsilon>0$.  We now consider $ \phi_k(r)\mathcal A_t  g_k$. Since $1<t<2$ and $|k|\le 1$, by Lemma \ref{lem:kernel}
 we have $ \phi_k(r)|\mathcal A_t  g_k|\lesssim \mathcal E^N_0\ast |g|$. Hence, it follows that 
 \[\Big\Vert r^{\frac{1}{q}}  \sup_{1<t<2}  |  \phi_k(r)\mathcal A_t  g_k | \Big\Vert_{L^q_{r,x_3}}  
   \lesssim \|g\|_p\]
   for $1\le p\le q\le \infty$. Therefore we get the desired estimate. 

\subsection{Global maximal estimate}\label{g-max} 
 Using the estimates in this section, one can  provide a simpler proof of the result due to Beltran et al. \cite{BGHS}, i.e., 
 \Be\label{global} \| r^\frac1p \sup_{0<t<\infty} |\mathcal A_t g | \|_{L^p_{r,x_3}}\le  C \|g\|_p\Ee
 for $2<p\le \infty$.    In order to show this we use the following lemma which is a consequence of Proposition  
 \ref{goal2}  and  \ref{goal4}. 
 
 \begin{lem} 
 Let $2\le p\le 4$. Then, for some $c>0$ we have 
 \Be  
\label{max-A1} 
\big\| r^\frac1p \sup_{1<t<2} |\mathcal A_t \pj g |  \big\|_{L^p_{r,x_3}}\le  C 2^{-cj} \|g\|_p\, .
\Ee
 \end{lem}

 \begin{proof}
 We briefly explain how one can show \eqref{max-A1}.   In fact,  similarly as before, we decompose 
\[ \mathcal A_t  \pj g= S_1+ S_3+ S_3+S_4,\]
where 
\[S_1:=  \sum_{k<-j}  \phi_k(r)\mathcal A_t \pj g, \quad  S_2:=  \sum_{-j\le k\le -2}  \phi_k(r)\mathcal A_t  \pj g, \quad  S_3:= \sum_{-1\le k\le 1}  \phi_k(r)\mathcal A_t \pj g, \] 
and $S_4= \mathcal A_t  \pj g - S_1-  S_2- S_3.$ 
Then, the estimate \eqref{max-A1} follows if we show  
$\| r^\frac1p \sup_{1<t<2} |S_\ell | \|_{L^p_{r,x_3}}\le  C 2^{-cj} \|g\|_p$, $\ell=1,2,3, 4$ for some $c>0$.   The estimate for $S_1$ follows from \eqref{ho} and summation over $k<-j$. 
Using the estimate of the second case in \eqref{goal2 ineq},  one can easily get the estimate for  $S_2$. 
The estimate for $S_3$  is obvious from Proposition  \ref{goal4}.  By Proposition \ref{goal2} combined with the localization property \eqref{localization}  we can obtain 
the estimate for $S_4$. However,  due to the projection operator $\pj$ we need to modify the previous  argument slightly. 

From \eqref{def-A1} and \eqref{def-A2} we see 
\begin{align}
\label{kj}
\mathcal A_t \pj g(r,x_3)=\iint  g(z_1, z_3)  
K_j\Big(\frac{r^2+t^2}{2}-z_1- t ry_1,x_3-z_2-try_2\Big)d\sigma(y) d z, 
 \end{align}
 where $K_j=\mathcal F^{-1}( \phi(2^{-j}|\cdot|)$. Note that $|K_j|\lesssim E_{-j}^N $  for any $N$ and $k\ge 2$. 
If $r\in \supp \phi_k$, $\sqrt{2z_1}\not\in  I_k $, and $k$ is large enough,  then  we have 
\[  \Big| K_j\Big( \frac{r^2+t^2}{2}-t ry_1 -z_1,x_3-try_2-z_2\Big)\Big|\lesssim  2^{-(2k+j)N} 
\Big(1+ 2^j|r^2-2z_1|+ 2^{-k}|x_3-z_2| \Big)^{-N}  \] 
for any $N$ since 
$|2^{-1}r^2-z_1|\gtrsim 2^{2k}$ and $|rty|\lesssim 2^k$. Hence it follows that 
\[ \| r^\frac1p \phi_k(r)\mathcal A_t \pj (1-\chi_{I_k})g\|_p\le C2^{-(k+j)N}  \|g\|_p, \quad 1\le p\le \infty\] 
for any $N$. We break $\mathcal A_t \pj g= \mathcal A_t \pj  \chi_{I_k}g+\mathcal A_t \pj  (1-\chi_{I_k})g$. Using the last inequality  and then  Proposition \ref{goal2},  we obtain  
\[ \| S_4\|_p\le  \Big( \sum_{k\ge  2  }     \| r^\frac1p \phi_k(r)\mathcal A_t \pj \chi_{I_k}g\|_p^p\Big)^\frac1p +  \sum_{k\ge  2  }  2^{-(k+j)N}  \|g\|_p \lesssim 2^{-cj} \|g\|_p \] 
for some $c>0$ by taking an $N$ large enough. 
 \end{proof}
 
    Once we have \eqref{max-A1}, using  a standard argument  which relies on the Littlewood-Paley decomposition and rescaling (for example, see \cite{B2, Schlag96, BGHS} ) 
one can easily show \eqref{global}. Indeed,  we break the maximal function into high and lower frequency parts:  
\[
 \sup_{0<t<\infty} |\mathcal A_t g|  \le  \mathcal A_{low\,} g  + \mathcal A_{high\,} g,
 \]
 where 
\begin{align*}
\mathcal A_{low\,} g &= \sup_l     \sup_{2^{l}\le t < 2^{l+1}} |\mathcal A_t \mathcal P_{<-2l} g|,
\\
\mathcal A_{high\,} g &=\sum_{k\ge 0} \sup_l     \sup_{2^{l}\le t < 2^{l+1}} |\mathcal A_t \mathcal P_{k-2l} g|.
\end{align*}

For $\mathcal A_{low\,} g$ we claim 
\Be \label{max-low}   \sup_{2^{l}\le t < 2^{l+1}} |\mathcal A_t \mathcal P_{<-2l} g(r, x_3)| \lesssim    \mathcal M_{\R^2}g(2^{-1}r^2,x_3). \Ee
This gives $\mathcal A_{low\,} g(r,x_3)\lesssim   \mathcal M_{\R^2}g(2^{-1}r^2,x_3)$. Since $\mathcal M_{\R^2}$ is bounded on $L^p$ for $p>2$,   for $2<p\le \infty$ we get 
\[\| r^\frac1p \mathcal A_{low\,} g \|_{L^p_{r,x_3}}\le  C \|g\|_p.\]
We now proceed to prove \eqref{max-low}. 
Note that $\sum_{j\le 2l}  \phi(2^{-j}|\cdot|)=\phi_{<1}(2^{2l}|\cdot|)$ and $\phi_{<1}$ is a smooth function supported on $[-2^2,2^2]$. Thus,  similarly as in  \eqref{kj} we note that
$\mathcal A_t \mathcal P_{<-2l} g(r,x_3) =  \iint  g(z_1, z_3)  
\widetilde {K}_l \ast d\sigma_{tr} ( 2^{-1}(r^2+t^2)-z_1,x_3-z_2)  d z $ where $\widetilde {K}_l=\mathcal F^{-1}( \phi_{<1}(2^{2l}|\cdot|))$. Since $\widetilde {K}_l\lesssim \mathcal E_{2l}^N$
for any $N$,   for $2^{l}\le t < 2^{l+1}$  we see 
   \Be 
   \label{att} |\mathcal A_t \mathcal P_{<-2l} g(r,x_3)|\lesssim  \Big|\int  | g(z_1, z_2) |
\mathcal E^{2N}_{2l} \ast d\sigma_{tr} \big(2^{-1}r^2- z_1,x_3-z_2\big) d z
\Ee
because $2^{2l}t^2\lesssim 1$ and $\mathcal E^{2N}_{2l}=2^{-4l}(1+ 2^{-2l}|y| )^{-2N}.$ Hence, taking an $N$ large enough,  we note  that  
\Be \mathcal E^{2N}_{2l} \ast d\sigma_{tr}(x) \lesssim   
\begin{cases}
\label{attt}
(2^{2l} tr)^{-1}  (1+   2^{-2l} ||x|- tr|)^{-N}, &    2^{2l}\ll   tr,
\\
2^{-4l} (1+   2^{-2l} |x|)^{-N}, &   2^{2l} \gtrsim   tr,
\end{cases}  
\Ee
provided that  $2^{l}\le t < 2^{l+1}$. Indeed, to show this we only have to consider the case $2^{2l}\ll tr$ since the other case is 
trivial. By scaling $x\to trx$  we may assume  that $tr=1$. Thus, it is enough  to show
$ \int L^{-2}(1+L^{-1}|x-y|)^{-2N}d\sigma(y)\lesssim L^{-1}(1+L^{-1}||x|-1|)^{-N}$
for $L\ll 1$ with an $N$ large enough. However, this is easy to see since  $|x-y|\ge ||x|-1|$ and  $\int L^{-1}(1+L^{-1}|x-y|)^{-N}d\sigma(y)\lesssim 1$.

 Therefore, combining \eqref{att} and \eqref{attt},  
one can  see
\[ \sup_{2^{l}\le t < 2^{l+1}}|\mathcal A_t \mathcal P_{<-2l} g(r,x_3)| \lesssim  M_{\R^2}g(2^{-1}r^2,x_3)  +  \mathfrak M_2g(2^{-1}r^2,x_3).\]
Here $\mathfrak M_2$ denotes the Hardy-Littlewood 
maximal function on $\mathbb R^2$. This proves the claim \eqref{max-low} since $\mathfrak M_2g\lesssim M_{\R^2}g$.  

So we are reduced to showing  $ \| r^\frac1p \mathcal A_{high\,} g \|_{L^p_{r,x_3}}\le  C \|g\|_p$ for $p>2$.  For the purpose it is  sufficient to show 
\Be
\label{hhh}  \|  \sup_{2^{l}\le t < 2^{l+1}} |\mathcal A_t \mathcal P_{k-2l} g|\|_p \lesssim 2^{-ck} \|g\|_p
\Ee
because  $\mathcal A_{high\,} g\le  \sum_{k\ge 0}  (\sum_l   | \sup_{2^{l}\le t < 2^{l+1}} |\mathcal A_t \mathcal P_{k-2l} g|^p)^{1/p}$ and $(\sum_l \| \mathcal P_{k-2l} g\|_p^p)^{1/p}\lesssim \|g\|_p.$   
By scaling, using \eqref{def-A}, we can easily see  the inequality \eqref{hhh} is equivalent to \eqref{max-A1} while $j$ replaced by $k$.  
 So,  we have \eqref{hhh} and this completes the proof of \eqref{global}.

\section{Proof of Proposition \ref{goal2}  and  \ref{goal4}}
\label{pf-pros}

In order to prove Proposition \ref{goal2}  and  \ref{goal4},  we are led  by  \eqref{def-A}  to consider  $\widehat{d\sigma}(tr\xi)$  
for which we use the following well known asymptotic expansion (see, for example,  \cite{Stein}):  
\Be
\label{asym}
\widehat{d\sigma}(\xi)=   \sum_{j=0}^NC^\pm_j   |\xi|^{-\frac12-j}     e^{ \pm i|\xi|}+ E_N( |\xi|),  \quad |\xi|\gtrsim 1
\Ee
where  $ E_N$ is a smooth function satisfying  
\Be 
\label{asym2}
| \frac {d^\ell}{dt^\ell}  E_N(r)|\lesssim r^{-N}
\Ee  for $0\le \ell\le 4$ if  $r\gtrsim 1$. 
The expansion \eqref{asym}  relates the operator $\mathcal A_t$ to the wave propagator.  
After changing variables, to prove Proposition  \ref{goal2}  and  \ref{goal4}    
we can use the local smoothing estimate for the wave operator (see Proposition \ref{local-pq} below).  

\smallskip

\subsection{Local smoothing estimate}
\label{ls}
 Let us denote 
\[  e^{it\sqrt{-\Delta}} f(x)=\frac1{(2\pi)^2}\int_{\mathbb R^2} e^{i(x\cdot\xi+t|\xi|)}\widehat{f}(\xi)d\xi.\] 
We make use of  $L^p$--$L^q$ local smoothing estimate for the wave equation in $\R^2$.

\begin{thm}
\label{local-pq} 
Let $j\ge 0$. Suppose \eqref{pq} holds. Then, for $\epsilon>0$ we have 
\Be
\label{local-smoothing} 
 \left\Vert    e^{it\sqrt{-\Delta}} \pj f  \right\Vert_{L^q_{x,t}(\R^2\times [1,2])}\lesssim 2^{\frac{3}{2}\left( \frac{1}{p}-\frac{1}{q}\right)  j+ \epsilon j}\Vert f\Vert_{L^q} 
 \Ee
\end{thm}
This follows by interpolating   the  estimates \eqref{local-smoothing} with $(p,q)=(2,2)$,  $(1,\infty)$, and  $(4,4)$. 
The estimate \eqref{local-smoothing} with $(p,q)=(2,2)$ is a straightforward consequence of Plancherel's theorem and 
\eqref{local-smoothing} with $(p,q)=(1,\infty)$ can be shown by the stationary phase  method (for example,  see \cite{L}). The 
case $(p,q)=(4,4)$   is due to Guth, Wang, and Zhang \cite{GWZ}. 
 
 From Theorem \ref{local-pq} we  can deduce the following estimate via simple rescaling argument.

\newcommand{\wv}{ e^{it\sqrt{-\Delta}} } 
\newcommand{\wvt}{ e^{i\tau\sqrt{-\Delta}} } 
\begin{lem}\label{scaled local smoothing}
Let $j\ge -\ell$. 
Suppose \eqref{pq} holds.  Then, for $\epsilon>0$ we have
\[ \left\Vert e^{it\sqrt{-\Delta}} \pj f  \right\Vert_{L^q_{x,t}(\R^2\times [2^\ell,2^{\ell+1}])}\lesssim 2^{\frac{3}{2}\left( \frac{1}{p}-\frac{1}{q} \right)(\ell+j)+\left(\frac{3}{q}-\frac{2}{p}\right)\ell+ \epsilon (\ell+j)}\Vert f\Vert_{L^p} .\]
\end{lem}
\begin{proof}
Changing variables $(x,t)\to 2^\ell(x,t)$,   we see 
\begin{align*}
 \left\Vert   e^{it\sqrt{-\Delta}} \pj f  \right \Vert_{L^q_{x,t}(\R^2\times [2^\ell,2^{\ell+1}])} = 2^{\frac{3\ell}{q}}\left\Vert   e^{it\sqrt{-\Delta}}  \mathcal P_{\ell+j} f(2^\ell\cdot) \right\Vert_{L^q_{x,t}(\R^2\times [1,2])}.
\end{align*}
Thus, using \eqref{local-smoothing} we have 
\[ \left\Vert   e^{it\sqrt{-\Delta}} \pj f  \right \Vert_{L^q_{x,t}(\R^2\times [2^\ell,2^{\ell+1}])} \lesssim  2^{\frac{3\ell}{q}+\frac{3}{2}\left( \frac{1}{p}-\frac{1}{q} \right)(\ell+j)+ \epsilon (\ell+j)}\Vert f(2^\ell\cdot)\Vert_{L^p}. \] 
So, rescaling gives the desired inequality. 
\end{proof}

\subsection{Proof of Proposition \ref{goal2}}
   We now recall \eqref{def-A} and \eqref{asym}.  To show Proposition \ref{goal2} 
we first deal with the contribution from the error part $E_N$.  Let us set 
 \[   \mathcal E_t  g  (r,x_3) =   \int e^{i(\frac{r^2+t^2}{2}\xi_1+x_3\xi_2)} E_N(tr|\xi|)\,   \widehat{g}(\xi) d\xi.  \]

\begin{lem}
\label{lem:err}
Let  $j\geq -k$. Suppose  \eqref{pq} holds. Then, we have
\begin{equation}\label{eq:err}
\Big\Vert \sup_{1<t<2}  | \phi_k(r)\mathcal E_t \mathcal{P}_j g |  \Big \Vert_{L^q_{r,x_3}}
\lesssim 
\begin{cases}
2^{-(N-3)(j+k)}2^{k(\frac{1}{q}-\frac{2}{p})}\Vert g\Vert_{L^p},  & \   k\ge -2,
\\[3pt]
2^{-(N-3)(j+k)}2^{k(\frac{3}{q}-\frac{2}{p})} \Vert g\Vert_{L^p}, & \  k< -2.
\end{cases}
\end{equation}
\end{lem}

\begin{proof}
We first consider the case $k\ge -2$. 
Using  Lemma \ref{sobo},  we  need to estimate $ \phi_k(r)\mathcal E_t \mathcal{P}_j g$ and $ \phi_k(r)\partial_t \mathcal E_t \mathcal{P}_j g$ in   $L^q_{r,x_3,t}(\mathbb R^2\times [1,2])$. 
For simplicity  we denote $L^q_{r,x_3,t}=L^q_{r,x_3,t}(\mathbb R^2\times [1,2])$. We first consider $ \phi_k(r)\mathcal E_t \mathcal{P}_j g$. 
Changing variables $\frac{r^2}{2}\mapsto s$, we note that
\[ \phi_k(\sqrt{2s})\mathcal E_t \mathcal{P}_j g(\sqrt{2s},x_3)=\phi_k(\sqrt{2s})\int \mathcal{K}\big( s-y_1+ 2^{-1}{t^2},x_3-y_2 \big)g(y_1,y_2)dy, \]
where
\begin{align*}
\mathcal{K}(s,u) = 2^{2j}  
\int e^{i2^j(s\xi_1+u\xi_2)}\phi_0(\xi)E_N(2^jt\sqrt{2s}|\xi|)d\xi.
\end{align*}
Since $s\sim 2^{2k}$, using \eqref{asym2}, we have
$|\mathcal{K}(s,u)|\lesssim 2^{2j}(1+2^j|(s,u)|)^{-M}2^{-N(j+k)} $
for  $1\le M\le 4$  via integration by parts. Thus,  we have 
$ \Vert \phi_k(\sqrt{2s})\mathcal{K}(s+\frac{t^2}{2}, u) \Vert_{L^r_{s,u}}\le C 2^{-N(j+k)}2^{2j(1-\frac{1}{r})} $
for  $1<t<2$ with a positive constant $C$.  Young's convolution inequality gives 
$ \| \phi_k(\sqrt{2s})\mathcal E_t \mathcal{P}_j g(\sqrt{2s},x_3)\|_{L^q_{s,x_3,t}}\lesssim  2^{-N(j+k)}2^{2j(\frac{1}{p}-\frac{1}{q})}\Vert g\Vert_{L^p} $. 
Thus, reversing $s\to r^2/2$,  after a simple manipulation we get 
\Be \label{k-large} \Big\Vert \phi_k(r)\mathcal E_t\mathcal{P}_j g  \Big\Vert_{L^q_{r,x_3,t}}\lesssim 2^{-(N-2)(j+k)} 2^{k(\frac{1}{q}-\frac{2}{p})}\Vert g\Vert_{L^p}\Ee
for $1\le p\le q\le \infty.$ 
We now consider $ \phi_k(r)\partial \mathcal E_t \mathcal{P}_j g$. Note that 
\Be  \label{err-d}
 \partial_t \mathcal E_t  g  (r,x_3) =   \int e^{i(\frac{r^2+t^2}{2}\xi_1+x_3\xi_2)}\big (t\xi_1 E_N(tr|\xi|)+     r|\xi|E_N'(tr|\xi|)\big)   \widehat{g}(\xi) d\xi. \Ee
Using \eqref{asym2}, we can handle $ \phi_k(r)\partial \mathcal E_t \mathcal{P}_j g$ similarly as before. In fact, since $|t\xi_1|\lesssim 2^j$ and $ r|\xi|\sim 2^{k+j}$, 
we see 
\[
\Big\Vert \phi_k(r) \partial_t \mathcal E_t \mathcal{P}_j g  \Big \Vert_{L^q_{r,x_3}}
\lesssim 2^{-(N-2)(j+k)}2^{k(\frac{1}{q}-\frac{2}{p})}(2^{j+k}+2^j)\Vert g\Vert_{L^p}. \] 
Hence, combining this and \eqref{k-large} with Lemma \ref{sobo}, we get \eqref{eq:err} for $k\ge -2$.

We now consider the case $k< -2$. We first claim that 
\Be 
\label{k-small}
 \Vert \phi_k(r)\mathcal E_t\mathcal{P}_j g\Vert_{L^q_{r,x_3,t}}\lesssim 2^{-(N-2)(j+k)}2^{k(\frac{2}{q}-\frac{2}{p})}\Vert g\Vert_{L^p}.
 \Ee
We use the transformation \eqref{chg}. By \eqref{jacobian} we have  $|\frac{\partial(y_1,y_2,\tau)}{\partial(r,x_3,t)}|\sim 1$. Therefore, 
\[ \Vert \phi_k(r)\mathcal E_t\mathcal{P}_j g\Vert_{L^q_{r,x_3,t}}\lesssim \Big( \int \Big| \phi_k(r(y,\tau)) \widetilde{K}(\cdot,\tau)\ast g(y) \Big|^qdy d\tau \Big)^{\frac{1}{q}},\]
where
\[ \widetilde{K}(y,\tau)=\int e^{iy\cdot\xi}\phi_j(\xi)E_N(\tau|\xi|)d\xi. \]
Note that $\tau\sim 2^k$. Changing $\tau\mapsto 2^k\tau$ and $\xi\mapsto 2^j\xi$, using \eqref{asym2} and integration by parts,  we have
$ |\widetilde{K}(y,2^k\tau)|\le C 2^{2j}(1+2^j|y|)^{-M}2^{-N(j+k)} $
for  $1\le M\le 4$ and $1<\tau<2$.  Young's convolution inequality gives
\[ \Vert \phi_k(r)\mathcal E_t\mathcal{P}_j g \Vert_{L^q_{r,x_3,t}}\lesssim 2^{-N(j+k)}2^{2j(\frac{1}{p}-\frac{1}{q})}\Vert g\Vert_{L^p}. \]
 Thus, we get \eqref{k-small}.
As for $ \phi_k(r)\partial \mathcal E_t \mathcal{P}_j g$, we use  \eqref{err-d} and repeat the same argument to see  $\Vert \phi_k(r)\partial_t \mathcal E_t\mathcal{P}_j g \Vert_{L^q_{r,x_3,t}}\lesssim 2^{-N(j+k)} 2^j 2^{2j(\frac{1}{p}-\frac{1}{q})}\Vert g\Vert_{L^p} $ since $|t\xi_1|\lesssim 2^j$, $ r|\xi|\sim 2^{k+j}$, and $k<-2$. Thus, we get 
\[ \Vert \phi_k(r)\partial_t \mathcal E_t\mathcal{P}_j g \Vert_{L^q_{r,x_3,t}} \lesssim 2^{-(N-2)(j+k)} 2^k 2^{k(\frac{2}{q}-\frac{2}{p})}\Vert g\Vert_{L^p}.\]
 Putting \eqref{k-small} and this together, by Lemma \ref{sobo}  we obtain  \eqref{eq:err} for $k<-2$. 
\end{proof}

By \eqref{asym} and Lemma \ref{lem:err}, 
to prove Proposition  \ref{goal2}  and  \ref{goal4} 
we only have to consider  contributions from the  remaining $ C^\pm_j   |tr\xi|^{-\frac12-j}     e^{ \pm i|tr\xi|},$  $j=0,\dots, N$. 
To this end, it is sufficient to consider
the major term $ C^\pm_0  |tr\xi|^{-\frac12}     e^{ \pm i|tr\xi|}$ since the other terms can be handled similarly.  
Furthermore,  by reflection $t\to -t$  it is enough to deal with  $|tr\xi|^{-\frac12}     e^{i|tr\xi|}$ since the estimate 
\eqref{local-smoothing}  clearly holds with the interval $[1,2]$ replaced by $[-2,-1]$.

 Let us set
\Be
\label{def-uk}
  \mathcal U_t g(r,x_3) =  \int e^{i(\frac{r^2+t^2}{2}\xi_1+x_3\xi_2+ tr|\xi|)}  |r\xi|^{-\frac12}    \widehat{g}(\xi) d\xi.  
  \Ee
To complete the proof of Proposition  \ref{goal2},  
we need to show 
\begin{equation}\label{k2}
\Big\Vert \sup_{1<t<2}  | \phi_k(r)  \mathcal U_t \pj g |  \Big \Vert_{L^q_{r,x_3}}
\lesssim 
\begin{cases}
2^{(j+k)(\frac{3}{2p}-\frac{1}{2q}-\frac{1}{2}+\epsilon)+\frac{k}{q}-\frac{2k}{p}}\Vert g\Vert_{L^p},  & \  \  \ k\ge 2,
\\[3pt]
2^{(j+k)(\frac{3}{2p}-\frac{1}{2q}-\frac{1}{2}+\epsilon)+\frac{2k}{q}-\frac{2k}{p}}\Vert g\Vert_{L^p}, & \  k\le -2.
\end{cases}
\end{equation}
Using  Lemma \ref{sobo},   the matter is reduced to obtaining estimates for $\phi_k(r)\mathcal U_t \pj g$ and  $\phi_k(r) \partial_t \mathcal U_t \pj g$ in  $L^q_{r,x_3,t}$.
Note that
\Be
\label{partial-uk}
 \partial_t\mathcal U_t \pj g (r,x_3,t)=\int e^{i(\frac{r^2+t^2}{2}\xi_1+x_3\xi_2+ tr|\xi|)}\widehat{\pj g}(\xi)\frac{t\xi_1+ r|\xi|}{|r\xi|^{1/2}} d\xi. 
 \Ee
By the Mikhlin  multiplier theorem one can easily see  
\[  \|  \phi_k(r) \partial_t\mathcal U_t \pj g\|_{L^q_{r,x_3,t}} \lesssim 
\begin{cases}
2^{j+k}\Vert   \phi_k(r) \mathcal U_t \pj  g \Vert_{L^q_{r,x_3,t}}, \quad &  k\ge 0,
\\
2^{j}\Vert   \phi_k(r) \mathcal U_t \pj g \Vert_{L^q_{r,x_3,t}}, \quad & k<  0,
\end{cases} 
\]
where $L^q_{r,x_3,t}$ denotes $L^q_{r,x_3,t}(\mathbb R^2\times [1,2])$. 
Therefore,  by  Lemma \ref{sobo} it is sufficient  for \eqref{k2} to prove that
\[ \Vert   \phi_k(r) \mathcal U_t \pj g \Vert_{L^q_{r,x_3,t}}\lesssim  \
\begin{cases}
 2^{(j+k)(\frac{3}{2p}-\frac{3}{2q}-\frac{1}{2}+\epsilon)+\frac{k}{q}-\frac{2k}{p}}\Vert g\Vert_{L^p},  \quad  & \  \  k\ge 2,
 \\
 2^{(j+k)(\frac{3}{2p}-\frac{3}{2q}-\frac{1}{2}+\epsilon)+\frac{3k}{q}-\frac{2k}{p}}\Vert g\Vert_{L^p} , \quad & k\le - 2.
 \end{cases}
 \]
We first consider the case $k\ge 2$. 
As before, we use the change of variables \eqref{chg}. Since $|\!\det\frac{\partial (y_1, y_2, \tau)}{\partial  (r,x_3,t)}|\sim 2^{2k}$ from \eqref{jacobian} and since $\tau=rt$ and $1<t<2$, we have
\[ \big\Vert    \phi_k(r)  \mathcal U_t \pj g\big\Vert_{L^q_{r,x_3,t}}\lesssim 2^{-\frac{2k}{q}-\frac{j+k}{2}}\big\Vert
e^{i\tau \sqrt{-\Delta}} \pj f  
\big\Vert_{{L^q_{y,\tau}(\R^2\times [2^{k-1},2^{k+2}])}} \]
since $|r\xi|\sim 2^{j+k}$.
Thus, Lemma \ref{scaled local smoothing} gives the desired estimate \eqref{k2} for $k\ge 2$. 
The case $k\le -2$ can be handled in the  exactly same manner. The only difference is that $|\!\det\frac{\partial (y_1, y_2, \tau)}{\partial  (r,x_3,t)}| \sim 1$. Thus, the desired estimate \eqref{k2}  immediately
follows from Lemma \ref{scaled local smoothing}.

\subsection{Proof of Proposition \ref{goal4}}\label{singular section}
As mentioned already, the determinant of Jacobian ${\partial (y_1, y_2, \tau)}/{\partial  (r,x_3,t)}$ may vanish when $|k|\le 1$. So, 
we need additional decomposition depending on   $|r-t|$. We also make decomposition in $\xi$ depending on $|\xi|^{-1}{\xi_1}+1$ to control the size of the multiplier $\left| t\xi_1+r|\xi| \right|$ more accurately  (for example,  
see \eqref{compare2}).

{For $m\ge 0$ let us set}
\begin{align*}  
\psi_m(\xi)&=\phi\big(2^m\big| |\xi|^{-1} \xi_1 +1\big |\big),
\\ 
\psi^m(\xi)&=1-\sum_{0\le j <m}\psi_j(\xi).
\end{align*}
so that  $\sum_{0\le k <m}\psi_k+ \psi^m=1$.  We additionally define 
\newcommand{\pjm}{\mathcal P_{\!j,m}}
\[  \mathcal P_{\!j,m} g=  ( \phi_j  \psi_m \widehat g)^\vee, \quad   \mathcal P_{\!j}^{m} g=  ( \phi_j  \psi^m\, \widehat g)^\vee. \] 
So it follows that 
\Be\label{proj-m} \pj  =\sum_{0\le k <m} \mathcal P_{\!j,k}+  \mathcal P_{\!j}^{m}. \Ee

\begin{prop}\label{singular local smoothing}  Let us set 
$\phi_{k,l}(r,t)= \phi_k(r)\phi(2^l|r-t|). $
Let  $j\geq -1$ and $k=-1,0,1$. Suppose \eqref{pq} holds. Then, 
if  $0\leq l\leq m/{2}$, for  $\epsilon>0$ we have
\Be
\label{l-small}
\Vert  \phi_{k,l} \mathcal U_t  \pjm g\Vert_{L^q_{r,x_3,t}}\lesssim  2^{-\frac 12 j} 2^{\frac{l}{q}} 2^{(\frac{m}{2}-l)(\frac{1}{p}+\frac{3}{q}-1)+\frac{3j}{2}(\frac{1}{p}-\frac{1}{q})+\epsilon j}\Vert g\Vert_{L^p}. 
\Ee
\end{prop}

\newcommand{\bx}{ \mathbf x }
\newcommand{\bt}{ \mathbf t }
In order to prove Proposition \ref{singular local smoothing}, we make the change of variables \eqref{chg}. 
Since $|k|\le 1$, we need only to consider $(r,t)$ contained in the set $[2^{-1} -10^{-2},2^2+10^2]\times[1,2]$. Set
\[  S_l=\big\{( y_1, y_2, \tau):   2^{-2l-1}\le |y_1-\tau|\le 2^{-2l+1}, \ y_1, \tau\in [2^{-3}, 2^3] \big\}.\]
By \eqref{chg}  $y_1-\tau=(r-t)^2/2$. From \eqref{jacobian} we  note $|\!\det\frac{\partial (y_1, y_2, \tau)}{\partial  (r,x_3,t)}|\sim 2^{-l}$ if $(y_1,\tau)\in S_l$.  Thus, 
changing variables $(r,x_3,t)\to  (y_1, y_2, \tau)$  we obtain 
\Be
\label{kl-wave}
\Vert  \phi_{k,l} \mathcal U_t  \pj h \Vert_{L^q_{r,x_3,t}}\lesssim    2^{-\frac 12 j}   2^{\frac{l}{q}}  \| \wvt  h \|_{L^q_{y,\tau}(S_l)}. 
\Ee
Therefore, for \eqref{l-small} 
it is sufficient to show 
\begin{align}
\label{l-sm}
\Vert   \wvt  \pjm g\|_{L^q_{y,\tau}(S_l)} &\lesssim   2^{(\frac{m}{2}-l)(\frac{1}{p}+\frac{3}{q}-1)+\frac{3j}{2}(\frac{1}{p}-\frac{1}{q})+\epsilon j}\Vert g\Vert_{L^p}
\end{align}
for $p,q$ satisfying \eqref{pq}. 
For the purpose we need the following lemma, which gives an improved $L^2$ estimate thanks to restriction of the integral over $S_l$. 
\begin{lem}\label{thin L2 estimate}
Let $D_l=\{(x_1,x_2,t):   2^{-2l}\le |x_1-t|\le 2^{-2l+1} \}$.  Then, we have
\Be \label{eq:thinl2}  \left\Vert \int e^{i(x\cdot \xi+t|\xi|)}\widehat{g}(\xi)\psi_m(\xi)d\xi \right\Vert_{L^2_{x,t}(D_l)}\lesssim 2^{\frac{m}{2}-l}\Vert g\Vert_{L^2}. \Ee
\end{lem}
\begin{proof}
We write
$ x\cdot \xi+t|\xi|=x_1(\xi_1+|\xi|)+ x_2\xi_2+(t-x_1)|\xi|. $ Then, changing  variables   $(x,t-x_1) \to (x, t)$ and  $\xi\to \eta:= \mathcal L(\xi)=(\xi_1+|\xi|,\xi_2),$ we see
\begin{align*}
\left\Vert \int e^{i(x\cdot \xi+t|\xi|)}\widehat{g}(\xi)\psi_m(\xi)d\xi \right\Vert_{L^2_{x,t}(D_l)} \leq \Big\Vert \int e^{i(x\cdot\eta+t|\mathcal{L}^{-1}\eta|)}\frac{\widehat{h}(\mathcal{L}^{-1}\eta)}{|\!\det J\mathcal{L}(\eta)|} d\eta \Big\Vert_{L^2_{x,t}(\R^2\times I_l)}
\end{align*}
where $\widehat{h}(\xi)=\widehat{g}(\xi)\psi_m(\xi)$ and $ I_l=[-2^{-2l+1},-2^{-2l}]\cup  [2^{-2l},2^{-2l+1}].$ By Plancherel's theorem, 
we have
\begin{align*} 
\left\Vert \int e^{i(x\cdot \xi+t|\xi|)}\widehat{g}(\xi)\psi_m(\xi)d\xi \right\Vert_{L^2_{x,t}(D_l)} \le 
C2^{-l}\Big\Vert \frac{\widehat{h}(\mathcal{L}^{-1}\cdot)}{|\!\det J\mathcal{L}|} \Big\Vert_{L^2_x}.
\end{align*}
A computation shows $\det J\mathcal{L}=1+|\xi|^{-1}{\xi_1}$, so $|\!\det J\mathcal{L}| \sim 2^{-m}$ on the support of $\widehat h$.  Thus, by changing variables and Plancherel's theorem we get  
\eqref{eq:thinl2}.  
\end{proof}

We also use the following  elementary lemma. 
\begin{lem}\label{freq restrict}
For any $1\leq p\leq \infty$, $j$, and $m$, we have
\begin{align*}
 \Vert ( \phi_j\psi_m \widehat{g}\,)^{\vee} \Vert_{L^p}\lesssim \Vert g \Vert_{L^p},  
 \quad 
  \Vert  ( \phi_j\psi^m \widehat{g}\,)^{\vee}\Vert_{L^p}\lesssim \Vert g \Vert_{L^p}. 
  \end{align*}
\end{lem}
\begin{proof}
Since $ \psi^m- \psi^{m+1}=\psi_m$, it suffices to prove the second inequality only. By Young's inequality we need only  to show
$\Vert (\phi_j\psi^m)^{\vee}\Vert_{L^1}\lesssim 1.$
By  scaling it is clear that  $\Vert (\phi_j(\xi)\psi^m(\xi))^{\vee}\Vert_{L^1}=\Vert (\phi_0(\xi)\psi^m(\xi))^{\vee}\Vert_{L^1}.$ 
Note that $\mathfrak m(\xi):=\phi_0(\xi)\psi^m(\xi)$ is supported in a rectangular box with dimensions $1\times 2^{-m}$. So,  
$\mathfrak m(\xi_1,   2^{-m}\xi_2)$ is supported in  a cube of side length $\sim 1$ and it is easy to  see $\partial_\xi^\alpha(\mathfrak m(\xi_1,   2^{-m}\xi_2))$ is uniformly bounded for any $\alpha$. 
This gives $\| (\mathfrak m(\cdot, 2^{-m}\cdot))^\vee \|_1\lesssim 1$. Therefore,  
after  scaling 
we get
$ \Vert (\phi_0(\xi)\psi^m(\xi))^{\vee}\Vert_{L^1}\lesssim 1. $
\end{proof}

\begin{proof}[Proof of \eqref{l-sm}]
In view of interpolation  the estimate \eqref{l-sm} follows for $p,q$ satisfying \eqref{pq}   if  we show the next three estimates:
\begin{align}
&\Vert  \wvt \pjm g \|_{L^2_{y,\tau}(S_l)} \lesssim 2^{\frac{m}{2}-l}\Vert g\Vert_{L^2},
\label{l2}
\\
& \Vert  \wvt \pjm g \|_{L^\infty_{y,\tau}(S_l)}\lesssim 2^{\frac{3j}{2}}\Vert g\Vert_{L^1}, 
\label{1toinfty}
\\
& \Vert  \wvt \pjm g \|_{L^4_{y,\tau}(S_l)} \lesssim 2^{\epsilon j}\Vert g\Vert_{L^4}.
\nonumber
\end{align}
 The first estimate follows from Proposition \ref{thin L2 estimate}.  Lemma \ref{scaled local smoothing} and Lemma \ref{freq restrict}  give the other two estimates. 
\end{proof}

It is possible to improve the estimate  \eqref{l-small} when  $j>m$.      
\begin{prop}\label{l>j/2 est}
Let $j\ge -1$ and $k=-1,0,1$. Suppose  $1\le p\leq q$, $1/{p}+1/{q}\leq 1$, and  $j>m$, then
\[
\Vert\phi_{k,l} \mathcal U_t \mathcal{P}_{j,m}g\Vert_{L^q_{r,x_3,t}}\lesssim 2^{-\frac{j}{2}}2^{\frac{l}{q}}2^{\frac{2}{q}(\frac{m}{2}-l)+\frac{j-m}{2}(1-\frac{1}{p}-\frac{1}{q})+\frac{3j}{2}(\frac{1}{p}-\frac{1}{q})}\Vert g\Vert_{L^p}.
\]
\end{prop}
\begin{proof}  
By \eqref{kl-wave} it is sufficient to show  
\[\Vert  \wvt \pjm g\Vert_{{L^q_{y,\tau}(S_l)}}\lesssim 2^{\frac{2}{q}(\frac{m}{2}-l)+\frac{j-m}{2}(1-\frac{1}{p}-\frac{1}{q})+\frac{3j}{2}(\frac{1}{p}-\frac{1}{q})}\Vert g\Vert_{L^p}\]
for $p,q$ satisfying  $1\le p\leq q$, $1/{p}+1/{q}\leq 1$. In fact, by  interpolation with the estimates  \eqref{l2} and \eqref{1toinfty}  we only have to show 
\begin{equation}\label{Linfty estimate}
\Vert  \wvt \pjm g \|_{L^\infty_{y,\tau}(S_l)} \lesssim 2^{\frac{j-m}{2}}\Vert g\Vert_{L^\infty}.
\end{equation}

Let us set 
\[  K^{j,m}_t(x)= \frac1{(2\pi)^2}\int e^{i(x\cdot\xi+t|\xi|)}\phi_j(|\xi|)\psi_m(\xi)d\xi.\] 
Then $   \wvt \pjm g  =   K^{j,m}_\tau \ast g .$
Therefore, \eqref{Linfty estimate} follows  if we show 
\begin{equation}\label{kernel estimate}  \|  K^{j,m}_t \| _{L^1_x}\lesssim 2^{\frac{j-m}{2}} \end{equation} 
when $t\sim 1$. 
Note that $ |\xi_2|/|\xi|=\sqrt{1-\xi_1/|\xi|}\sqrt{1+\xi_1/|\xi|}\lesssim 2^{-\frac{m}{2}} $ if $\xi\in \supp \psi_m$. 
So, $\supp \psi_m$ is contained in a conic sector with angle $\sim 2^{-\frac{m}{2}}$. Let $\mathcal{S}$ be a sector centered at the origin in $\R^2$ with angle $\sim 2^{-\frac{j}{2}}$ and $\phi_\mathcal{S}$ be a cut-off function adapted to $\mathcal{S}$. Then, by integration by parts it follows that 
\[ \Big
\Vert \int e^{i(x\cdot\xi+t|\xi|)}\phi_j(|\xi|)\phi_\mathcal{S}(\xi)d\xi \Big\Vert_{L^1_x}\lesssim 1 \] 
if $t\sim 1$.
(See, for example, \cite{L}). Now \eqref{kernel estimate} is clear since the support of $\psi_m$ can be decomposed into  as many as $C2^{\frac{j-m}{2}}$ such sectors. 
\end{proof}

Finally,  we prove Proposition \ref{goal4}  making use of Proposition \ref{singular local smoothing} and \ref{l>j/2 est}.
We recall  \eqref{def-A} and \eqref{asym}. 
As  mentioned  before, by  Lemma \ref{lem:err}  we need only to consider  
$\mathcal U_t$ (see \eqref{def-uk}) and it is sufficient to show
\Be
\label{uk-est}
\big \Vert \sup_{1<t<2}|  \phi_k(r) \mathcal U_t \pj g|  \big\Vert_{L^q_{r,x_3}}\lesssim 2^{ \frac12 (\frac{3}{p}-\frac{1}{q}-1)j  +\epsilon j}\Vert g\Vert_{L^p}
 \Ee
for $p,q$ satisfying $p\leq q$, $1/{p}+1/{q}< 1$ and $1/{p}+2/{q}> 1$.  

\begin{proof}[Proof of \eqref{uk-est}]
Let us set $ \phi^l(\cdot)=1-\sum_{j=0}^{l-1} \phi (2^j\cdot)$ and $\phi_k^l(r,t)=\phi_k(r)\phi^l(|r-t|)$. Then, we decompose 
\[\phi_k(r)=\sum_{0\le l\le j/2} \phi_{k,l}(r,t)  +  \sum_{j/2<  l< j}^{} \phi_{k,l}(r,t)  +  \phi^j_k(r,t). \]
Combining this with \eqref{proj-m} and using $ \sum_{\frac{j}{2}<l< j} \phi_{k,l}\le   \phi_{k}^{[j/2]-1}$, by the triangle inequality we have
\begin{align*}
&\big\Vert  \sup_{1<t<2}|  \phi_k(r) \mathcal U_t \pj g| \big\Vert_{L^q}  \le \sum_{\ell=1}^5 S_\ell,
\end{align*}
where 
\begin{align*}
S_1&= \sum_{0\le l\le j/2}\,  \sum_{0\le m\le l-1} \big \Vert \sup_{1<t<2} \phi_{k,l} |\mathcal U_t \pjm g |\big\Vert_{L^q},
\quad 
S_2=\sum_{0\le l\le j/2}\Vert\sup_{1<t<2}\phi_{k,l} |\mathcal U_t \mathcal{P}_j^l g| \Vert_{L^q},
\\
S_3&=\sum_{\frac{j}{2}<l< j}\,\sum_{0\le m\le j-1}\Vert \sup_{1<t<2} \phi_{k,l} |\mathcal U_t \pjm g |\Vert_{L^q} ,
\quad  \
S_4= \sum_{0\le m\le j-1} \Vert   \sup_{1<t<2} \phi_{k}^j |\mathcal U_t \pjm g | \Vert_{L^q},
\\
&\qquad \qquad \qquad \qquad \qquad  S_5=\Vert  \sup_{1<t<2} \phi_{k}^{[j/2]-1}  |\mathcal U_t \pj^jg | \Vert_{L^q}.
\end{align*}
The proof of \eqref{uk-est}  is now reduced to showing 
\Be
\label{ss-}
S_\ell \lesssim 2^{ \frac12 (\frac{3}{p}-\frac{1}{q}-1)j  +\epsilon j}\Vert g\Vert_{L^p}, \quad 1\le \ell\le 5, 
 \Ee
for $p,q$ satisfying $p\leq q$, $1/{p}+1/{q}< 1$ and $1/{p}+2/{q}> 1$.

We first consider $S_1$. Using Lemma \ref{sobo}, we need to estimate  $\phi_{k,l} \mathcal U_t \pjm g $  and   $\partial_t (\phi_{k,l} \mathcal U_t \pjm g)$ in  $L^q_{r,x_3,t}(\mathbb R^2\times [1,2])$. Writing $t\xi_1+r|\xi|= t\big( |\xi|^{-1} \xi_1 +1\big )+(r-t)$, we note that
\Be
\label{compare2}
 | t\xi_1+r|\xi|| 
\lesssim 2^j\max\{ 2^{-m}, 2^{-l}\}. \Ee
Note that $\partial_t  \phi_{k,l}=O(2^l)$ and $2^l\lesssim 2^{j-m}$. Thus,  recalling \eqref{partial-uk}, we apply Lemma \ref{sobo} and the Mikhlin multiplier theorem to get
\[  S_1\lesssim \sum_{0\le l\le j/2}  \sum_{m=0}^{l-1} 2^{\frac{j-m}{q}} \big \Vert \phi_{k,l} \mathcal U_t \pjm g \big\Vert_{L^q}.\] 
Thus, by Proposition \ref{singular local smoothing}  it follows
\begin{align*}
S_1
& \lesssim 2^{-\frac{j}{2}+\frac{j}{q}+\frac{3j}{2}(\frac{1}{p}-\frac{1}{q})+\epsilon j} \sum_{0\le l\le j/2}  2^{l(1-\frac{1}{p}-\frac{2}{q})}\sum_{m=0}^{l-1}2^{\frac{m}{2}(\frac{1}{p}+\frac{1}{q}-1)}\Vert g\Vert_{L^p}.
\end{align*}
Since $1/{p}+1/{q}-1< 0$ and $1/p+2/q>1$,  we obtain \eqref{ss-} with $\ell=1$. 

We now estimate $S_3$, which can be handled similarly.   Since  $\partial_t  \phi_{k,l}=O(2^l)$ and $2^l\gtrsim 2^{j-l}$, using \eqref{compare2} and \eqref{partial-uk},  we see \[ S_3 \lesssim  \sum_{\substack{j/2<l< j}}\sum_{ 0\leq m\leq j-1} 2^{\frac{l}{q}}\Vert \phi_{k,l} \mathcal U_t \pjm g \Vert_{L^q}\] by applying Lemma \ref{sobo} and the Mikhlin multiplier theorem.  Thus, using Proposition \ref{l>j/2 est}, we get  \eqref{ss-} with $\ell=3$  since  $1/p+2/q>1$. 

We can show the estimate \eqref{ss-} with $\ell=2$ in the same manner. As before, since  $\partial_t  \phi_{k,l}=O(2^l)$ and $2^l\lesssim  2^{j-l}$, using \eqref{compare2} and applying Lemma \ref{sobo} and the Mikhlin multiplier theorem,  we have 
\[ S_2\lesssim \sum_{0\le l\le j/2}   2^{\frac{j-l}{q}} \big \Vert \phi_{k,l} \mathcal U_t \pj^l g \big\Vert_{L^q}.\]  Thus, by \eqref{kl-wave} and Proposition \ref{local-pq}, we have
$S_2\lesssim \sum_{0\le j\le \frac{j}{2}}   2^{-\frac{j}{2}}  2^{\frac{j}{q}+\frac{3j}{2}(\frac{1}{p}-\frac{1}{q})+\frac\epsilon 2 j}\Vert g\Vert_{L^p}, 
$
which gives   \eqref{ss-} with $\ell=2$. 
 
We handle $S_4$ and $S_5$ without  relying on Lemma \ref{sobo}. Instead, we control $S_4$ and $S_5$ more directly. 
Concerning $S_4$  we claim that
\Be\label{s4} 
S_4\lesssim 2^{ \frac12 (\frac{3}{p}-\frac{1}{q}-1)j }\Vert g\Vert_{L^p} 
\Ee
if  ${5}/{q}> 1+{1}/{p}$ and $2\le p\le q\le \infty$.  This clearly gives \eqref{ss-} with $\ell=4$  for $p,q$ satisfying $p\leq q$, $1/{p}+1/{q}< 1$ and $1/{p}+2/{q}> 1$. We note that 
\[ |\phi_{k}^j \mathcal U_t \pjm g (r,x_3)|\lesssim   2^{-\frac 12 j}\Big|\phi_{k}^j \int e^{i2^j(r^2\xi_1+x_3\xi_2+ r^2|\xi|)} \mathfrak m(\xi)\phi_0(\xi)\psi_m(\xi)\widehat{g(2^{-j}\cdot)}(\xi)d\xi\Big|,  \]
where 
\[  \mathfrak m(\xi)= e^{i2^j(\frac{t^2-r^2}{2}\xi_1+ (t-r)r|\xi|)}|\xi|^{-\frac{1}{2}} \widetilde \phi_0(\xi), \] 
and $\widetilde\phi_0$ is a smooth function supported in $[-\pi, \pi]^2$ such that $\widetilde\phi_0\phi_0=1$. If $(r,t)\in \supp \phi_{k}^j$, then $|t-r|\lesssim 2^{-j}$. 
Thus, $|\partial_\xi^\alpha m(\xi)|\lesssim 1$ for any $\alpha$. Expanding $\mathfrak m$ into Fourier series on  $[-\pi, \pi]^2$ we have 
$ \mathfrak m(\xi)= \sum_{k\in \mathbb Z^2} C_{\mathbf k} (r,t)   e^{i\mathbf k\cdot\xi}$ 
while $|C_{\mathbf k} (r,t)|\lesssim (1+|\mathbf k|)^{-N} $.  Therefore, after scaling $\xi\to 2^j\xi$, 
the estimate  \eqref{s4} follows if we obtain  
\[
\Vert \mathcal R \pjm  g\Vert_{L^q_{r,x_3}([2^{-2}, 2^3]\times \mathbb R)} \lesssim  2^{ \frac12 (\frac{3}{p}-\frac{1}{q})j }\Vert g\Vert_{L^p},
\]
where 
\[ \mathcal R  g(r,x_3)=\int e^{i(r^2\xi_1+x_3\xi_2+ r^2|\xi|)} \widehat{g}(\xi)d\xi.\]
When $q=2$, changing variables $r^2\to r$ and following the argument in  the proof of Lemma \ref{thin L2 estimate} we have
$\Vert \mathcal R \pjm  g\Vert_{L^2_{r,x_3}([2^{-2}, 2^3]\times \mathbb R)} \lesssim 2^{{m}/{2}}\Vert g\Vert_{L^2}.$
On the other hand,   \eqref{Linfty estimate} gives
$\Vert \mathcal R \pjm g\Vert_{L^\infty_{r,x_3}([2^{-2}, 2^3]\times \mathbb R)}\lesssim 2^{(j-m)/{2}}\Vert g\Vert_{L^\infty}.$
Interpolation  between these two estimates gives
\[ \Vert \mathcal R\pjm g\Vert_{L^q_{r,x_3}([2^{-2}, 2^3]\times \mathbb R)} \lesssim 2^{\frac{m}{q}+\frac{j-m}{2}(1-\frac{2}{q})}\Vert g\Vert_{L^q} \]
for $2\leq q\leq \infty$. 
Since the support $\widehat{\pjm g}(\xi)$ is contained in a rectangular region of dimensions  $2^j\times  2^{j-\frac{m}{2}}$, by Bernstein's inequality  we have
\[ \Vert \mathcal R_m^j  g\Vert_{L^q_{r,x_3}([2^{-2}, 2^3]\times \mathbb R)}\lesssim 2^{j(\frac{2}{p}-\frac{3}{q})+m(\frac{5}{2q}-\frac{1}{2}-\frac{1}{2p})}\Vert g\Vert_{L^p}\]
for $2\le p\le q\le \infty$. Since  ${5}/{q}> 1+{1}/{p}$, this proves the claimed estimate \eqref{s4}.

Finally, we show \eqref{ss-} with $\ell=5$. Changing variables $(\xi_1, \xi_2)\to (2^j\xi_1, \xi_2)$, we observe 
 \[ \phi_{k}^{[j/2]-1} | \mathcal U_t \pj^j g (r,x_3)|\lesssim   2^{\frac j2}\phi_{k}^{[j/2]-1} 
\Big|\int e^{i(\frac{(r-t)^2}2  2^j\xi_1+x_3\xi_2)} \mathfrak m(\xi)\widehat{\mathcal P_j^j g}(2^j\xi_1, \xi_2)d\xi\Big|,  \]
where 
\[  \widetilde{\mathfrak m}(\xi)= e^{i2^j rt( |(\xi_1, 2^{-j}\xi_2)|-\xi_1)}|(\xi_1, 2^{-j}\xi_2)|^{-\frac{1}{2}} \widetilde \phi_0(|(\xi_1, 2^{-j}\xi_2)|)  \psi^{j-1}(2^j\xi_1, \xi_2).\]  
Note that $ \supp \widetilde{\mathfrak m}\subset \{ \xi_1\in [2^{-1}, 2^2] , |\xi_2| \le 2^2  \}$. Since 
$|\partial_\xi^\alpha m(\xi)|\lesssim 1$ for any $\alpha$,  expanding $\widetilde{\mathfrak m}$ into Fourier series on  $[-2\pi, 2\pi]^2$ we have
$ \widetilde{\mathfrak m}(\xi)= \sum_{k\in \mathbb Z^2} C_{\mathbf   k} (r,t)   e^{i 2^{-1}\mathbf k\cdot\xi}$ 
while $|C_{\mathbf k} (r,t)|\lesssim (1+|\mathbf k|)^{-N} $.  Hence, similarly as before, changing variables $(\xi_1, \xi_2)\to (2^{-j}\xi_1, \xi_2)$, 
to show  \eqref{ss-} with $\ell=5$  it is sufficient to obtain  
\Be
\label{3.22}
\Big\Vert \sup_{1<t<2} \mathcal{P}_j^j g\Big(\frac{(r-t)^2}{2},x_3\Big) \Big\Vert_{L^q_{r,x_3}([2^{-2}, 2^3]\times \mathbb R)} 
\lesssim 2^{ \frac12 (\frac{3}{p}-\frac{1}{q})j }\Vert g\Vert_{L^p}
\Ee
for $1\le p\le q\le \infty$. 
Clearly, the left hand side is bounded by  $\Vert \mathcal{P}_j^j g(x_1,x_3)\Vert_{L^q_{x_3}(L^\infty_{x_1})}$.  
The Fourier transform of $ \mathcal{P}_j^j g$ is supported on the rectangle $\{ \xi_1\in [2^{j-1}, 2^{j+2}] , |\xi_2| \le 2^{j+2}  \}$.  
Thus, using  Bernstein's inequality in $x_1$, we get 
\begin{align*}
\Big\Vert \sup_{1<t<2}\mathcal{P}_j^j g\Big(\frac{(r-t)^2}{2},x_3\Big) \Big\Vert_{L^q_{r,x_3}([2^{-2}, 2^3]\times \mathbb R)} 
& \lesssim 2^{-\frac{j}{2}+\frac{j}{q}}\Vert \mathcal{P}_j^j g\Vert_{L^q} 
\end{align*}
for $1\le q\le \infty$. So,  another use of  Bernstein's inequality gives \eqref{3.22} for $1\le p\le q\le \infty$.
 This completes the proof of  \eqref{uk-est}.
\end{proof}

\noindent{\bf Acknowledgement.}  Juyoung Lee was supported by the National Research Foundation of Korea (NRF) grant no. 2017H1A2A1043158  and Sanghyuk Lee  was supported  by NRF grant no.  2021R1A2B5B02001786.  

\bibliographystyle{plain}

\end{document}